\documentclass[12pt]{amsart}
\usepackage{amssymb}

\usepackage{graphicx}
\usepackage{epsfig}

\newtheorem{theorem}{Theorem}
\theoremstyle{plain}
\newtheorem{corollary}[theorem]{Corollary}
\newtheorem{definition}[theorem]{Definition}
\newtheorem{proposition}[theorem]{Proposition}

\newcommand\bR{\mathbb{R}}
\newcommand\bE{\mathbb{E}}

\def\dd{\mathrm{d}}
\def\mm{\mathrm{m}}

\newcommand\bld[1]{\boldsymbol{#1}}
\newcommand\opr[1]{\mathrm{#1}}

\numberwithin{theorem}{section}
\numberwithin{equation}{section}

\begin{document}

\today


\title[]{Small Ball Probabilities for the Infinite-Dimensional Ornstein-Uhlenbeck
Process in Sobolev Spaces}
\author{S. V. Lototsky}
\curraddr[S. V. Lototsky]{Department of Mathematics, USC\\
Los Angeles, CA 90089 USA\\
tel. (+1) 213 740 2389; fax: (+1) 213 740 2424}
\email{lototsky@usc.edu}
\urladdr{http://www-bcf.usc.edu/$\sim$lototsky}


\begin{abstract}
While small ball, or lower tail, asymptotic for Gaussian measures
generated by solutions of stochastic ordinary differential equations is
relatively well understood, a lot less is known in the case of
stochastic partial differential
equations. The paper presents  exact logarithmic asymptotics
of the small ball probabilities in a scale of Sobolev spaces
when the Gaussian measure is generated by the
solution of a diagonalizable stochastic parabolic equation.
   Compared to the finite-dimensional case, new effects appear in
a certain  range of the Sobolev exponents.
\end{abstract}

\maketitle

\section{Introduction}

A standard Gaussian random variable $\zeta$ is very unlikely to be large:
$$
\mathbb{P}(|\zeta|^2>A) \leq e^{-A/2}, \ A>0,
$$
(cf. \cite[Lemma A.3]{Khosh-SPDE}), but it is relatively likely to be small:
by direct computation,
$$
\mathbb{P}(|\zeta|^2\leq \varepsilon)\geq \frac{\varepsilon^{1/2}}{3},\ \varepsilon<1.
$$
In fact, for every finite collection of iid standard Gaussian random variables
$\zeta_1, \ldots, \zeta_n$, \\
analysis of the
density of the $\chi^2_{n}$ distribution shows that
\begin{equation*}
\label{SB-FD}
\lim_{\varepsilon\to 0} \varepsilon^{-n/2}\mathbb{P}\left(\sum_{k=1}^n \zeta_k^2\leq \varepsilon\right)=\frac{2^{(2-n)/2}}{n\Gamma(n/2)},
\end{equation*}
where $\Gamma$ is the Gamma function.
Similarly, for {\em finitely} many Gaussian random variables,
the asymptotic of
\begin{equation*}
\label{FinSum}
\mathbb{P}\left(\sum_{k=1}^n a_k\zeta_k^2\leq \varepsilon\right),\
a_k>0,
\end{equation*}
is always algebraic in $\varepsilon$, as $\varepsilon\to 0$.
On the other hand,  for a standard $N$-dimensional Brownian motion $\bld{w}=\bld{w}(t)$, $0\leq t\leq T$,
$$
\mathbb{P}\left(\int_0^T |\bld{w}(t)|^2dt \leq \varepsilon \right)
\approx e^{-N^2T^2/(8\varepsilon)}, \ \ \ \varepsilon \to 0,
$$
that is,
\begin{equation}
\label{SB-sBM}
\lim_{\varepsilon\to 0} \varepsilon
 \ln \mathbb{P}\left(\int_0^T  |\bld{w}(t)|^2 dt \leq \varepsilon \right) = -\frac{N^2T^2}{8};
\end{equation}
cf. \cite[Theorem 6.3 and Corollary 3.1]{lishao01} or Corollary \ref{cor:NBM}
below.

Transition from a finite to an infinite number of Gaussian random variables
 typically leads to  a qualitative change of behavior of small ball (or lower tail)
 probabilities: if $a_k>0,\ \sum_{k} a_k<\infty,$
then,  as $\varepsilon \to 0$, the decay of
\begin{equation*}
\label{InfSum}
\mathbb{P}\left(\sum_{k=1}^{\infty}a_k \zeta_k^2\leq \varepsilon\right),\ \
\end{equation*}
is usually faster that polynomial in $\varepsilon.$

The logarithmic asymptotic \eqref{SB-sBM} is rather robust:
if  $\bld{x}=\bld{x}(t)$ is the solution of the linear equation
\begin{equation}
\label{OU-fd}
d\bld{x}(t)=A\bld{x}(t)+Qd\bld{w}(t), \ 0<t\leq T, \ \bld{x}(0)=0,
\end{equation}
with a positive-definite matrix $Q$, then
\begin{equation}
\label{SB-OU-fd}
\lim_{\varepsilon\to 0} \varepsilon \ln \mathbb{P}
\left(\int_0^T |\bld{x}(t)|^2dt\leq \varepsilon \right) =
-\frac{\big(T\;\mathrm{trace}(Q)\big)^2}{8}
\end{equation}
regardless of the matrix $A$; cf. \cite[Theorem 4.5]{SVL-MM}.
In one-dimensional case, \eqref{SB-OU-fd} continues to hold
even with some time-dependent drifts \cite{fatalov08}.

A possible  infinite-dimensional generalization of \eqref{OU-fd}
 is the stochastic wave equation
\begin{equation}
\label{SWE}
u_{tt}=u_{xx}+g(u)+\dot{W}(t,x),\ t>0,\ x\in \bR,
\end{equation}
where $W=W(t,x) $ is a two-parameter Brownian sheet and $\dot{W}(t,x)
=\partial^2 W/(\partial t \partial x)$
is the corresponding space-time Gaussian white noise. Indeed, a
change of variables reduces \eqref{SWE} to
\begin{equation}
\label{OU-2pm}
\frac{\partial^2 v}{\partial t\partial x}= g(v)+
\frac{\partial^2 \tilde{W}}{\partial t\partial x},
\end{equation}
with a different Brownian sheet $\tilde{W}$; cf. \cite[Theorem 3.1]{Walsh}.
 Equation
\eqref{OU-2pm} can thus be considered  an extension of \eqref{OU-fd} to two
independent variables in the spirit of
\cite[Section 7.4.2]{Khosh}; according to  \cite{SmallBall-StochWave}, the small
ball probabilities for $u$ and $W$  have similar asymtotics.

So far, the paper \cite{SmallBall-StochWave} appears to be the only
work  addressing the question of small ball probabilities
for stochastic partial differential equations.
The objective of the current paper is to investigate the asymptotic
 behavior of
 $\ln \mathbb{P}(\|u\|^2_{L_2((0,T);H^{\gamma})}\leq \varepsilon),$ $\varepsilon \to 0$,
 where $u$ is the solution of the stochastic parabolic  equation
\begin{equation}
\label{SHE}
u_t(t,\bld{x})+\opr{A} u(t,\bld{x}) =
\dot{W}(t,\bld{x}), 0<t\leq T,\ \bld{x} \in G,
\end{equation}
 $\opr{A}$ is a positive self-adjoint  elliptic operator on a bounded domain
 $G\subset \bR^{\dd}$,
 $\dot{W}$ is space-time Gaussian white noise,
 and $H^{\gamma},\ \gamma\in \bR,$ is the
scale of Sobolev space generated by $\opr{A}$. An expansion of the solution
of \eqref{SHE} in eigenfunctions of $\opr{A}$ leads to an infinite system of
ordinary differential equations, making \eqref{SHE}  an
infinite-dimensional version of \eqref{OU-fd}. The results can be
summarized as follows: For both $X=u$ and $X=W$,
as $\varepsilon\to 0$,
$$
 \ln \mathbb{P}\left(
 \int_0^T  \|X(t)\|_{-\gamma}^2 dt \leq \varepsilon \right)
  \sim
  \begin{cases}
  \displaystyle - \frac{T^2}{8}
\mathfrak{C}(\gamma)
\varepsilon^{-1},& {\rm \ if \ } \gamma>\dd,\\
& \\
\displaystyle  -\mathfrak{C}(T,X)
 \varepsilon^{-1}\,
  |\ln \varepsilon|^2,& {\rm \ if \ } \gamma=\dd,\\
  & \\
 \displaystyle  -\mathfrak{C}(T,\gamma,X)\, \varepsilon^{-\varpi(\gamma,X)},
 & {\rm \ if \ } \gamma_0(X)< \gamma<\dd,
 \end{cases}
$$
where $\mathfrak{C}$, $\gamma_0$, and
$\varpi$ are suitable numbers. For example, $\gamma_0(W)=\dd/2$ and
$\varpi(\gamma,W)=\dd/(2\gamma-\dd)$. In particular, if $\gamma > \dd$, then
the result  is very similar to the finite-dimensional case
\eqref{SB-OU-fd}. The details are below in  Theorem \ref{prop:cBM-SB} ($X=W$)
and Theorem \ref{th:main} $(X=u)$.

Throughout the paper, for $f(x)>0$, $g(x)>0$, the notation
 $$
 f(x)\sim g(x),\ x\to x_0,
 $$
  means
 $$
 \lim_{x\to x_0} \frac{f(x)}{g(x)}=1,
 $$
 $f(x)=O(g(x)), \ x\to x_0,$ means $\limsup_{x\to x_0} f(x)/ g(x)<\infty,$
 and $f(x)=o(g(x)), \ x\to x_0,$ means $\lim_{x\to x_0}
 f(x)/g(x)=0$.
 The variable $x$ can be discrete or continuous and the limiting
 value $x_0$ finite or infinite.  We also
  fix
  $\big(\Omega, \mathcal{F} ,
  \{\mathcal{F}_t\}_{0\leq t\leq T}, \mathbb{P}\big)$, a stochastic basis
satisfying the usual assumptions.

\section{ Background on small ball probabilities}

Let $\zeta_n,\ n\geq 1,$ be independent identically distributed standard Gaussian
random variables and let $a_n,\ n\geq 1,$ be  positive
real  numbers such that $\sum_n a_n <\infty$.
By direct computation,
\begin{equation}
\label{LPT0}
\begin{split}
\bE \exp\left( -p \sum_{n=1}^{\infty} a_n \zeta_n^2 \right)=
&
\prod_{n=1}^{\infty} \bE e^{-pa_n\zeta_n}
= \prod_{n=1}^{\infty} \left(1+2p\,a_n\right)^{-1/2}\\
&=\exp\left(-\frac{1}{2}\sum_{n=1}^{\infty} \ln (1+2pa_n)\right), \ p>0,
\end{split}
\end{equation}
 and then {\tt Tauberian theorems} make it possible to connect the asymptotic of the right-hand side of \eqref{LPT0}  as $p\to +\infty$ with the asymptotic of
$$
\mathbb{P}\left(\sum_{n=1}^{\infty} a_n \zeta_n^2\leq \varepsilon\right)
$$
as $\varepsilon\to 0$.
The most general result in this direction was obtained in  \cite{Syt1}:
\begin{equation}
\label{eq:Sytaya}
\begin{split}
\mathbb{P}\left(\sum_{n=1}^{\infty} a_n \zeta_n^2
\leq \varepsilon\right)
&\sim
\left(4\pi \sum_{n\geq 1} \left(\frac{a_n \mathfrak{r}(\varepsilon)}
{1+2a_n\mathfrak{r}(\varepsilon)}\right)^2\right)^{-1/2}
\\
& \times
\exp\left(\varepsilon \mathfrak{r}(\varepsilon)-\frac{1}{2}
\sum_{n\geq 1}
\ln \big(1+2a_n\mathfrak{r}(\varepsilon)\big)\right),\ \varepsilon\to 0.
\end{split}
\end{equation}
 The function  $\mathfrak{r}=\mathfrak{r}(\varepsilon)$ is defined
 implicitly by the relation
 $$
\varepsilon
=\sum_{n\geq 1} \frac{a_n}{1+2a_n\mathfrak{r}(\varepsilon)},
$$
and this implicit dependence on $\varepsilon$ is the main
drawback of  \eqref{eq:Sytaya} in concrete applications.

Less precise but more explicit bounds are possible using
 {\tt exponential Tauberian theorems}, such as
 Theorems \ref{thm:tauberian} and \ref{thm:tauberian1} below;
  they are
 modifications of \cite[Theorem 3.5]{lishao01}
(which,  in turn, is a modification of \cite[Theorem 4.12.9]{RegVar}).

\begin{theorem}
\label{thm:tauberian}
Let $\xi$ be a non-negative random variable.
Then
\begin{align*}
\ln \big(\mathbb{E} e^{-p \xi}\big)
 \sim  -\alpha p^{\tau},\ p\to +\infty,\
 {\rm \ \ for \ some \ } \ \ \alpha>0,\
0<\tau<1,
\end{align*}
holds if and only if
\begin{align}
\label{eq:asydev}
\ln \mathbb{P}(\xi\leq\varepsilon) \sim
-\big((1-\tau) \alpha\big)^{1/(1-\tau)}
\left(\frac{\tau}{1-\tau}\right)^{{\tau}/(1-\tau)}\,
\varepsilon^{-\tau/(1-\tau)},\ \varepsilon \to 0.
\end{align}
\end{theorem}

While \eqref{eq:asydev} is only  {\em logarithmic}
asymptotic of the probability
and is not as strong as \eqref{eq:Sytaya}, it is usually more
explicit than \eqref{LPT0} and is good enough in many
applications.

When \eqref{eq:asydev} holds, we
say that the random variable $\xi$ has the {\tt small ball rate}
 $$
\varpi=\frac{\tau}{1-\tau}
$$
and the {\tt small ball constant}
 $$
\mathfrak{C}=\left((1-\tau)\alpha\right)^{1/(1-\tau)}
\left(\frac{\tau}{1-\tau}\right)^{{\tau}/(1-\tau)}.
$$

Occasionally, a more refined version of Theorem \ref{thm:tauberian}
is necessary.

\begin{theorem}
\label{thm:tauberian1}
Let $\xi$ be a non-negative random variable.
Then
\begin{equation*}
\ln \big(\mathbb{E} e^{-p \xi}\big)
 \sim  -\alpha p^{\tau}(\ln p)^{\beta},\ p \to +\infty,
 {\rm \ \ for \ some \ } \ \ \alpha>0,\, \beta>0,\ 0<\tau<1,
\end{equation*}
holds if and only if
\begin{align*}
\ln \mathbb{P}(\xi\leq\varepsilon) \sim
-\big((1-\tau) \alpha\big)^{1/(1-\tau)}
\left(\frac{\tau}{1-\tau}\right)^{{\tau}/(1-\tau)}
\varepsilon^{-\tau/(1-\tau)} |\ln \varepsilon |^{\beta/(1-\tau)}.
\end{align*}
\end{theorem}

\begin{proposition}
Let $x=x(t)$ be the solution of the equation
$$
dx(t)=-ax(t)dt+\sigma dw(t),\ 0<t<T,
$$
with $a\in \bR$ and $\sigma>0$,
and assume that the initial condition $x(0)$  is independent of
the Brownian motion $w$ and is a Gaussian random variable
 with  mean $\mu_0$ and variance
$\sigma_0^2$. Then
 \begin{equation}
 \label{OU-Psi-IC}
\bE \exp\left( -p\int_0^T x^2(t)dt\right)
=
\left(\frac{e^{aT}}{\cosh(\varrho T)+(a/\varrho)\sinh(\varrho T)}\right)^{1/2}
\times \frac{\exp\left(-\frac{\psi \mu_0^2}{1+2\sigma_0^2 \psi}\right)}{\sqrt{1+2\sigma_0^2 \psi}},
\end{equation}
 where
\begin{align*}
 \varrho&=(a^2+2\sigma^2 p)^{1/2},\\
  \psi &=\frac{\varrho-a}{2\sigma^2}
\left( 1-\frac{e^{-\varrho T}}{\cosh(\varrho T)
+(a/\varrho)\sinh(\varrho T)}\right).
\end{align*}
\end{proposition}

\begin{proof} This follows by direct computation using
\cite[Theorem 3]{Koncz}; see also
\cite[Lemma 17.3]{LSh-II} when  $\mu_0=\sigma_0=0$.
\end{proof}

\begin{corollary}
For the standard Brownian motion, with $a=\mu_0=\sigma_0=0$, and
$\sigma=1$, equality \eqref{OU-Psi-IC} becomes the
well-known Cameron-Martin formula:
\begin{equation}
\label{CMF}
\bE \exp\left( -p\int_0^T w^2(s)ds\right) =
\Big(\cosh(\sqrt{2p}\ T)\Big)^{-1/2}.
\end{equation}
\end{corollary}

As an illustration of Theorem \ref{thm:tauberian}, let us
confirm  \eqref{SB-sBM}.

\begin{corollary}
\label{cor:NBM}
If $\bld{w}=\bld{w}(t)$,
$0\leq t\leq T,$ is an $N$-dimensional standard Brownian motion, then
\begin{equation}
\label{SB-sBM-1}
 \ln \mathbb{P}\left(\int_0^T  |\bld{w}(t)|^2 dt \leq \varepsilon \right)
  \sim -\frac{N^2T^2}{8}\, \varepsilon^{-1},\ \varepsilon\to 0.
\end{equation}
\end{corollary}

\begin{proof}
By \eqref{CMF} and independence of the components of $\bld{w}$,
\begin{equation*}
\label{CMF-limN}
\ln  \bE \exp\left( -p\int_0^T |\bld{w}(t)|^2dt\right) \sim
-\frac{NT}{\sqrt{2}}\, \sqrt{p}.
\end{equation*}
Then \eqref{SB-sBM-1}  follows from Theorem \ref{thm:tauberian} with $\alpha=NT/\sqrt{2}$ and
$\tau=1/2$.
\end{proof}

\section{Diagonalizable Stochastic Parabolic Equation}
\label{sec:DSPE}

Let $\opr{A}$ be a positive-definite self-adjoint elliptic operator of order
$2\mm$ on a bounded domain $G\subset \bR^{\dd}$ with sufficiently
smooth boundary; alternatively, $G$ can be a smooth closed $\dd$-dimensional
manifold with smooth measure $d\bld{x}$.
Denote by $\lambda_k,\ k\geq 1,$ the eigenvalues of $\opr{A}$, and
by $\varphi_k,\ k\geq 1,$ the corresponding normalized eigenfunctions.
Our main assumption is that the {\tt Weyl-type asymptotic} holds
for $\lambda_k$:
\begin{equation}
\label{Weyl}
\lambda_k=\mathfrak{S}\,
k^{2\mm/\dd}\big(1+O(k^{-1/\dd})\big),\ k\to \infty,
\end{equation}
with constant $\mathfrak{S}$ depending only on the region $G$;
see \cite[Theorem 1.2.1]{SafVas}. For example,
 if $\opr{A}=-\bld{\Delta}$ on $G\subset\bR^{\dd}$
 with zero boundary conditions, and $|G|$ is
 the Lebesgue measure of $G$, then $\mm=1$ and
 $$
 \mathfrak{S}=
  4\pi 
  \left(\frac{\Gamma\left(1+\frac{\dd}{2}\right)}{|G|}\right)^{2/\dd}.
 $$

For $f\in \mathcal{C}_{0}^{\infty}(G)$
[that is, $f$ is a smooth compactly supported real-valued function on $G$]
 and $\gamma\in \bR$, define
\begin{equation*}
\label{SobNorm}
 f_k=\int_G f(\bld{x})\varphi_k(\bld{x})d\bld{x} \ \ {\rm and \ } \
\|f\|_{\gamma}^2=\sum_{k=1}^{\infty} \lambda_k^{\gamma/\mm}f_k^2.
\end{equation*}
Then define the space $H^{\gamma}$ as the closure of $\mathcal{C}_0^{\infty}(G)$
with respect to the norm $\|\cdot\|_{\gamma}$. In particular,
$$
H^0=L_2(G),\
H^{\gamma} = \opr{A}^{\gamma/(2\mm)} \big(L_2(G)\big),\
\|\opr{A}^sf\|_{\gamma}=\|f\|_{\gamma+2\mm s}.
$$
We also define
\begin{equation*}
\label{notations}
\begin{split}
&\langle f, g \rangle =
\int_G (\opr{A}^{-\gamma} f)(\bld{x})(\opr{A}^{\gamma} g)(\bld{x})
d\bld{x},\ f\in H^{\gamma},\ g\in H^{-\gamma};\\
&H^{\infty} =\bigcap_{\gamma}  H^{\gamma},\
H_{\infty} =\bigcup_{\gamma} H^{\gamma},
\end{split}
\end{equation*}
and identify $f\in H^{\gamma}$ with a possibly divergent series
\begin{equation*}
\label{f-sum}
f=\sum_{k\geq 1} f_k \varphi_k,\ f_k=\langle f, \varphi_k \rangle.
\end{equation*}
Note that $\varphi_k\in H^{\infty}$ for all $k\geq 1$.

\begin{definition}
A cylindrical Brownian motion on $L_2(G)$ is a Gaussian process
$\bld{W}=\bld{W}(t,f)$, indexed by $t\in [0,T]$ and $f\in L_2(G)$,
such that
$$
\bE  \bld{W}(t,f) = 0,\ \bE \big(\bld{W}(t,f)\bld{W}(s,g)\big) = \min(t,s)\int_{G}f(\bld{x})g(\bld{x})d\bld{x}.
$$
\end{definition}

\begin{proposition}
\label{prop:WN}
Define
$$
w_k(t)=\bld{W}(t,\varphi_k)
$$
and
 \begin{equation}
 \label{cbm-sum}
 W(t)=\sum_{k=1}^{\infty} w_k(t)\varphi_k.
 \end{equation}
 Then, for  $\gamma>\frac{\dd}{2}$,
 \begin{equation}
\label{sbm-reg}
W\in L_2\Big(\Omega\times (0,T);H^{-\gamma}\Big),\ \
 \end{equation}
 and
 \begin{equation}
 \label{cBM-gm}
 \bld{W}(t,f)=\langle W(t), f\rangle,\ f\in H^{\gamma}.
 \end{equation}
 \end{proposition}

 \begin{proof}
 Definition of $\bld{W}$ implies that
 $w_k(t)$ are independent standard Brownian motions. Then
 both \eqref{sbm-reg} and \eqref{cBM-gm} follow by
 direct computation.
In particular,  by \eqref{Weyl},
$$
\bE\|W(t)\|^2_{-\gamma} = t\sum_{k=1}^{\infty}
\lambda_k^{-\gamma/\mm}<\infty
$$
 if and only if $\gamma>\dd/2$.
Moreover,
$$
\bE \big| \langle W(t), f \rangle \big|^2 = t\|f\|_0^2,
$$
so that \eqref{cBM-gm} extends to $f\in L_2(G)$, and,
 by Kolmogorov's criterion, $W$ has a modification in
 $L_2\Big(\Omega;\mathcal{C}\big((0,T);H^{-\gamma}\big)\Big)$,
 $\gamma>\frac{\dd}{2}$.
 \end{proof}

 We then get the following analogue of \eqref{SB-sBM-1}.
\begin{theorem}
\label{prop:cBM-SB}
As $\varepsilon\to 0$,
\begin{equation}
\label{SB-cylBM-1}
 \ln \mathbb{P}\left(\int_0^T  \|W(t)\|_{-\gamma}^2 dt \leq \varepsilon \right)
  \sim
  \begin{cases}
 \displaystyle
  - \frac{T^2}{8}
\left(\sum_{k=1}^{\infty} \lambda_k^{-\gamma/(2\mm)}\right)^2\
\varepsilon^{-1},& {\rm if} \ \gamma>\dd,\\
 & \\
\displaystyle
-\frac{\mathfrak{S}^{-\dd/\mm}T^2}{32} \, \varepsilon^{-1}\, |\ln \varepsilon|^2,& {\rm if} \ \gamma=\dd,\\
& \\
\displaystyle  - \mathfrak{C}_{\gamma}
\varepsilon^{-\varpi},& {\rm if} \ \frac{\dd}{2}<\gamma<\dd,
\end{cases}
\end{equation}
where
\begin{align*}
\varpi &=\frac{\dd}{2\gamma-\dd},\
\mathfrak{C}_{\gamma}=(2q-1)^{2q}q^{-2q\varpi}2^{(1-4q)\varpi}
(T\mathfrak{S}^{-\gamma/(2\mm)})^{2\varpi}C_{\gamma}^{2q\varpi},\\
q&=\frac{\gamma}{\dd},\ {C}_{\gamma}=\int_0^{+\infty}
\frac{\ln \cosh(y)}{y^{1+(\dd/\gamma)}}\ dy.
\end{align*}
\end{theorem}

\begin{proof}
By \eqref{CMF} and \eqref{cbm-sum},
\begin{align}
\notag \bE \exp\left(-p \int_0^T\|W(t)\|_{-\gamma}^2 dt\right)&=
\prod_{k=1}^{\infty}\bE
\exp\left(-p \lambda_k^{-\gamma/\mm}\int_0^T\|w_k(t)\|^2 dt\right)
\\
\label{cBM-p}
& = \exp\left(-\frac{1}{2}\sum_{k=1}^{\infty}
\ln \cosh\left(T\sqrt{2p\lambda_k^{-\gamma/\mm}}\;\right)\right),
\end{align}
and by \eqref{Weyl},
$$
\lambda_k^{-\gamma/\mm}\sim \mathfrak{S}^{-\gamma/\mm} k^{-2\gamma/\dd},\
k\to \infty.
$$

If $\gamma>\dd$, then the series $\sum_k \lambda_k^{-\gamma/(2\mm)}$
converges, and the dominated convergence theorem implies
$$
\lim_{p\to \infty} p^{-1/2}\sum_{k=1}^{\infty}
\ln \cosh\left(T\sqrt{2p\lambda_k^{-\gamma/\mm}}\;\right)
={T}{\sqrt{2}}\left(
\sum_{k=1}^{\infty} \lambda_k^{-\gamma/(2\mm)}
\right),
$$
or
$$
-\frac{1}{2}\sum_{k=1}^{\infty}
\ln \cosh\left(T\sqrt{2p\lambda_k^{-\gamma/\mm}}\;\right)
\sim -\frac{T}{\sqrt{2}}\left(
\sum_{k=1}^{\infty} \lambda_k^{-\gamma/(2\mm)}
\right)
\ \sqrt{p},\ \ p\to +\infty,
$$
so that  the first relation in
\eqref{SB-cylBM-1} follows from Theorem \ref{thm:tauberian}.

If $\frac{\dd}{2}< \gamma \leq \dd$, then  we establish the  asymptotic of
\eqref{cBM-p} by comparison with a suitable integral.

Note that
\begin{equation*}
\label{aux-A}
T\sqrt{2p\lambda_k^{-\gamma/\mm} }= \sqrt{p} f(k),
\end{equation*}
and the function $f=f(x)$ satisfies
$$
f(x)=A_{\gamma} x^{-\gamma/\dd} (1+O(x^{-1/\dd})), \ x\to +\infty,
$$
with $A_{\gamma}=2^{1/2}T\mathfrak{S}^{-\gamma/(2\mm)}$.
Let $q=\gamma/\dd$.
By direct computation, as $p\to +\infty$,
\begin{align*}
\sum_{k=1}^{\infty}&
\ln \cosh\big(T\sqrt{2p\lambda_k^{-\gamma/\mm}}\big)
\sim
\int_1^{\infty} \ln  \cosh \big(\sqrt{p}f(x)\big)dx  \\
&\sim q^{-1}(A_{\gamma}\sqrt{p})^{1/q}
\int\limits_0^{A_{\gamma}\sqrt{p}}
 \frac{\ln \cosh (y)}{y^{1+(1/q)}}\,dy
\sim
\begin{cases}
2^{-1}A_{\gamma} \ \sqrt{p}\,\ln p,& \ {\rm if } \ q=1;\\
 &\\
q^{-1}A_{\gamma}^{1/q} C_q\, p^{1/(2q)},&\ {\rm if } \ \frac{1}{2}<q<1.
\end{cases}
\end{align*}
After that, Theorems \ref{thm:tauberian1} and
  \ref{thm:tauberian} imply the remaining relations in \eqref{SB-cylBM-1}.

\end{proof}

We now use the process $W$ to construct an infinite-dimensional analogue
of \eqref{OU-fd}.

Given $r>0$, consider the equation
\begin{equation}
\label{PE}
\dot{u}(t)+\opr{A}^ru(t)=\dot{W}(t),\ 0<t\leq T,\ u(0)=0.
\end{equation}
For example, with
\begin{equation*}
\begin{split}
 &G=[0,\pi],\\
 & \opr{A}=-\partial^2/\partial x^2, \ {\rm  \ zero \ boundary \
 conditions},\\
 & \lambda_k=k^2,\ \varphi_k(x)=(2/\pi)^{1/2} \sin(kx).
 \end{split}
 \end{equation*}
 and  $r=1$,
equation \eqref{PE} becomes
\begin{equation}
\label{example-PE}
u_t=u_{xx}+\dot{W}(t,x), \ 0<t\leq T,\ 0<x<\pi,\ u(0,x)=u(t,0)=u(t,\pi)=0.
\end{equation}

\begin{definition}
The solution of the equation \eqref{PE}
is a mapping from $\Omega\times [0,T]$ to $H_{\infty}$
with the following properties:
\begin{enumerate}
\item There exists a $\gamma\in \bR$ such that
$
u\in L_2\Big(\Omega; \mathcal{C}\big((0,T); H^{\gamma}\big)\Big).
$
\item For every $h\in H^{-\gamma}$,
the process $\langle u(t), h\rangle $, $0\leq t\leq T$,
 is $\mathcal{F}_t$-adapted.
\item For every $h\in H^{\infty}$, the equality
\begin{equation}
\label{PE-sol}
\langle u(t), h\rangle + \int_0^t \langle u(s), \opr{A}^{r}h\rangle ds =
\langle W(t),h\rangle
\end{equation}
holds in $L_2\big(\Omega\times(0,T)\big).$
\end{enumerate}
\end{definition}

\begin{proposition}
\label{prop:HE}
Equation \eqref{PE} has a unique solution $u=u(t)$.
Moreover, for every $\gamma>\frac{d}{2}$,
\begin{equation}
\label{PE-reg1}
u\in L_2\Big(\Omega; L_2\big((0,T); H^{-\gamma+r\mm}\big)\Big)
\bigcap L_2\Big(\Omega; \mathcal{C}\big((0,T); H^{-\gamma}\big)\Big).
\end{equation}
\end{proposition}

\begin{proof} The result can be derived from general
existence and uniqueness theorems for stochastic evolution equations,
such as  \cite[Theorem 5.4]{DaPr1-2} or \cite[Theorem 3.1.1]{Roz};
below is an outline of a direct proof.

Taking $h=\varphi_k$ in \eqref{PE-sol} we find
\begin{equation}
\label{PE-OUk}
\dot{u}_k(t)=-\lambda_k^{r}u_k(t)+\dot{w}_k(t),\ u_k(0)=0,
\end{equation}
that is,
$$
u_k(t)=\int_0^t e^{-\lambda_k^r(t-s)}dw_k(s).
$$
Then
$$
\bE u_k^2(t)=\frac{1-e^{-2\lambda_k^rT}}{2\lambda_k^r}
$$
and
$$
\bE\int_0^T \|u(t)\|^2_{-\gamma+r\mm}dt \leq
T\sum_{k=1}^{\infty} \lambda_k^{-\gamma/\mm}.
$$
By \eqref{Weyl},
$$
\lambda_k^{-\gamma/\mm}\sim \mathfrak{S}^{-\gamma/\mm} \, k^{-2\gamma/{\dd}},\
k\to \infty.
$$
If $\gamma>\frac{\dd}{2}$, then $-2\gamma/{\dd}<-1$, and
\eqref{PE-reg1} follows.

Similarly,
$$
\bE\big(u_k(t)-u_k(s)\big)^2\leq |t-s|,
$$
and then   Kolmogorov's criterion implies that
$u$ has a modification in $L_2\Big(\Omega; \mathcal{C}\big((0,T); H^{-\gamma}\big)\Big)$.

To establish uniqueness, note that the difference $v$ of two
solutions satisfies the deterministic
 equation $\dot{v}+\opr{A}v=0$ with zero initial condition.
\end{proof}

\begin{theorem}
\label{th:main}
As $\varepsilon\to 0$,
\begin{equation}
\label{SB-heat-1}
 \ln \mathbb{P}\left(
 \int_0^T  \|u(t)\|_{-\gamma}^2 dt \leq \varepsilon \right)
  \sim
  \begin{cases}
  \displaystyle - \frac{T^2}{8}
\left(\sum_{k=1}^{\infty} \lambda_k^{-\gamma/(2\mm)}\right)^2\
\varepsilon^{-1},& {\rm \ if \ } \gamma>\dd,\\
& \\
\displaystyle  -\frac{\mathfrak{S}^{-\dd/\mm}T^2}{32}\,
  \left(\frac{\dd}{\dd+2r\mm}\right)^2 \, \varepsilon^{-1}\,
  |\ln \varepsilon|^2,& {\rm \ if \ } \gamma=\dd,\\
  & \\
 \displaystyle  -\mathfrak{C}_{\gamma,\mm}\, \varepsilon^{-\varpi},
 & {\rm \ if \ } \frac{\dd}{2}-r\mm< \gamma<\dd,
 \end{cases}
\end{equation}
where
\begin{align}
\notag
&\mathfrak{C}_{\gamma,\mm}=\frac{\big((1-\tau)\mathfrak{S}^{-\dd/(2\mm)}T
C_{\gamma,\mm}\big)^{1/(1-\tau)}}{2}\, (\varpi)^{\varpi},\\
\notag
&  \tau=\frac{2r\mm+\dd}{4r\mm+2\gamma},\ \ \varpi=
   \frac{\tau}{1-\tau}=\frac{2r\mm +\dd}{2\gamma+2r\mm-\dd},\\
   \label{PE-SB-C}
&C_{\gamma,\mm}= \int_0^{\infty}
\frac{dy}
{y^{2(r\mm+\gamma)/\dd}+\sqrt{y^{4(r\mm+\gamma)/\dd}
+y^{2\gamma/\dd}}}.
\end{align}
\end{theorem}

\begin{proof}
Define
\begin{equation}
\label{pr-AB}
A_k=\sqrt{\lambda_k^{2r}+2p\lambda_k^{-\gamma/\mm}},\ \ \
B_k=\frac{\lambda_k^r}{A_k}.
\end{equation}
By \eqref{OU-Psi-IC} and \eqref{PE-OUk},
\begin{equation}
\label{pr-1}
\begin{split}
&\bE \exp\left( - p \int_0^T \|u(t)\|_{-\gamma}^2dt\right)
=\prod_{k=1}^{\infty} \bE \exp\left(-p\lambda_k^{-\gamma/\mm}
\int_0^T u_k^2(t)dt\right)\\
&=\exp\left(
\frac{T}{2}\sum_{k=1}^{\infty}  (\lambda_k^r-A_k)
-\frac{1}{2}\sum_{k=1}^{\infty} \ln \left(\frac{1+B_k}{2}\right)
-\frac{1}{2}\sum_{k=1}^{\infty} \ln \left(1+\frac{1-B_k}{1+B_k}e^{-2A_k}
\right)\right)\\
& := \exp\big(-S_1(p)+S_2(p)-S_3(p)\big),
\end{split}
\end{equation}
where
\begin{align}
\label{pr-S1-def}
S_1(p)&=\frac{T}{2}\sum_{k=1}^{\infty}  (A_k-\lambda_k^r),\\
\label{pr-S2-def}
S_2(p)&=-\frac{1}{2}\sum_{k=1}^{\infty} \ln \left(\frac{1+B_k}{2}\right),\\
\label{pr-S3-def}
S_3(p)&=\frac{1}{2}\sum_{k=1}^{\infty}
\ln \left(1+\frac{1-B_k}{1+B_k}e^{-2A_k}\right).
\end{align}
The goal is to show that, as $p\to \infty$,
\begin{equation}
\label{MainAsymptotic}
-\ln\bE \exp\left( - p \int_0^T \|u(t)\|_{-\gamma}^2dt\right)
\sim S_1(p)
\end{equation}
and
\begin{equation}
\label{S1p-main}
S_1(p)
\sim
\begin{cases}
\displaystyle
\left(\frac{T}{\sqrt{2}}\sum_{k=1}^{\infty}\lambda_k^{-\gamma/(2\mm)}\right)
\ {p}^{1/2},& {\rm \ if \ } \gamma>\dd,\\
 & \\
\displaystyle
\frac{T\mathfrak{S}^{-\dd/(2\mm)}}{2^{3/2}}\frac{\dd}{\dd+2r\mm}\
p^{1/2}\, \ln p,& {\rm \ if \ } \gamma=\dd,\\
& \\
\displaystyle
\alpha\, p^{\tau}
& {\rm \ if \ } \
\frac{\dd}{2}-r\mm<\gamma<\dd,
\end{cases}
\end{equation}
where 
$$
\tau=\frac{2r\mm + \dd}{4r\mm+2\gamma},\quad
\alpha=
T\mathfrak{S}^{-\dd/(2\mm)}2^{-(2r\mm+2\gamma-\dd)/(4r\mm+2\gamma)}
C_{\gamma,\mm};
$$
after that, relations \eqref{SB-heat-1} will follow from
Theorems \ref{thm:tauberian}
and \ref{thm:tauberian1}.

We start by establishing \eqref{S1p-main}.
 We then show that $S_2(p)$ and $S_3(p)$  are of lower order
 compared to $S_1(p)$:
 \begin{equation}
 \label{domin}
 S_j(p)=o\big(S_1(p)\big),\ j=2,3,\ p\to \infty.
 \end{equation}

Note that
\begin{equation}
\label{ld-A}
A_k - \lambda_k^r=
\frac{2p\lambda_k^{-\gamma/\mm}}{\lambda_k^r+\sqrt{\lambda_k^{2r}
+2p\lambda_k^{-\gamma/\mm}}}\, ,
\end{equation}
and recall that \eqref{Weyl} holds. If $\gamma>\dd$, then
$$
\sum_{k=1}^{\infty}\lambda_k^{-\gamma/(2\mm)}<\infty.
$$
Since,  for every $k\geq 1$,
$$
\lim_{p\to \infty}
\frac{2\sqrt{p}
\lambda_k^{-\gamma/\mm}}
{\lambda_k^r+\sqrt{\lambda_k^{2r}+2p\lambda_k^{-\gamma/\mm}}}
=\sqrt{2}\lambda_k^{-\gamma/(2\mm)},
$$
 the dominated convergence theorem implies
$$
\lim_{p\to \infty}\sum_{k=1}^{\infty}
\frac{2\sqrt{p}
\lambda_k^{-\gamma/\mm}}
{\lambda_k^r+\sqrt{\lambda_k^{2r}+2p\lambda_k^{-\gamma/\mm}}}
=\sqrt{2}\sum_{k=1}^{\infty}\lambda_k^{-\gamma/(2\mm)}.
$$
 Therefore, for $\gamma>\dd$,
\begin{equation*}
\label{OU-g>d}
S_1(p) \sim
\left(\frac{T}{\sqrt{2}}\sum_{k=1}^{\infty}\lambda_k^{-\gamma/(2\mm)}\right)
\sqrt{p},\ p\to \infty.
\end{equation*}

When $\frac{\dd}{2}-r\mm<\gamma\leq \dd$,   define the function
$$
f(x;p)=
\frac{1}
{x^{2(r\mm+\gamma)/\dd}+\sqrt{x^{4(r\mm+\gamma)/\dd}
+2p\,\mathfrak{S}^{-2r-(\gamma/\mm)}\,x^{2\gamma/\dd}}},\ x>0.
$$
By \eqref{ld-A},
$$
S_1(p)=pT\sum_{k=1}^{\infty} \frac{1}{\lambda_k^{r+(\gamma/\mm)}
+\sqrt{\lambda_k^{2r+2(\gamma/\mm)}+2p\lambda_k^{\gamma/m}}},
$$
whereas \eqref{Weyl} implies
\begin{equation}
\label{S1-f}
\frac{1}{\lambda_k^{r+(\gamma/\mm)}
+\sqrt{\lambda_k^{2r+2(\gamma/\mm)}+2p\lambda_k^{\gamma/m}}}
=\mathfrak{S}^{-r-(\gamma/\mm)}f\big(k+\epsilon(k);p\big),\
\end{equation}
with $\epsilon(k)=O(k^{1-(1/\dd)}),\ k\to \infty$,
uniformly in $p$.

Define
\begin{equation*}
\label{def-Sf}
S_f(p)=pT\mathfrak{S}^{-r-(\gamma/\mm)} \int_1^{\infty} f(x;p)dx.
\end{equation*}
We will now show that
\begin{equation}
\label{eq:same}
\lim_{p\to \infty} \frac{S_1(p)}{S_f(p)}=1.
\end{equation}
To begin, let us establish the asymptotic of
$S_f(p)$. With the notations
\begin{equation*}
\label{inf-OU-notations}
\nu=\frac{\dd}{2(\gamma+2r\mm)}, \ \ R=2p\,\mathfrak{S}^{-2r-(\gamma/\mm)},\
y=x R^{-\nu},
\end{equation*}
\begin{equation}
\label{OU-inf-int-M}
\begin{split}
S_f(p)&=
T\mathfrak{S}^{-\dd/(2\mm)}2^{-(2r\mm+2\gamma-\dd)/(4r\mm+2\gamma)}
p^{(2r\mm+\dd)/(4r\mm+2\gamma)}\\
&
\times
\int_{R^{-\nu}}^{\infty}
\frac{dy}
{y^{2(r\mm+\gamma)/\dd}+\sqrt{y^{4(r\mm+\gamma)/\dd}
+y^{2\gamma/\dd}}}.
\end{split}
\end{equation}

If $\gamma=\dd$, then \eqref{OU-inf-int-M} implies
\begin{equation*}
S_f(p) =
T\mathfrak{S}^{-\dd/(2\mm)}2^{-1/2}p^{1/2}
\int_{R^{-\nu}}^{+\infty}\frac{dy}{y^{2(r\mm+\dd)/\dd}+
\sqrt{y^{4(r\mm+\dd)/\dd}+y^2}}.
\end{equation*}
By L'Hospital's rule, for every $\kappa>0,$
$$
\lim_{\epsilon\to 0} \frac{1}{|\ln \epsilon|}\int_{\epsilon}^{\infty}
\frac{dy}{y^{1+\kappa}+\sqrt{y^{2+2\kappa}+y^2}}=
\lim_{\epsilon\to 0} \frac{\epsilon}{\epsilon^{1+\kappa}+\sqrt{\epsilon^{2+2\kappa}+\epsilon^2}}
=1.
$$
Therefore, as $p\to \infty$,
\begin{equation*}
\label{infOU-log}
S_f(p)\sim
\frac{T\mathfrak{S}^{-\dd/(2\mm)}}{2^{3/2}}\frac{\dd}{\dd+2r\mm}\
p^{1/2}\, \ln p.
\end{equation*}

If $\frac{\dd}{2}-r\mm<\gamma<\dd$, then \eqref{OU-inf-int-M} implies
\begin{equation*}
\label{infOU-power}
\begin{split}
S_f(p)
\sim
T\mathfrak{S}^{-\dd/(2\mm)}2^{-(2r\mm+2\gamma-\dd)/(4r\mm+2\gamma)}
C_{\gamma,\mm}\, p^{(2r\mm+\dd)/(4r\mm+2\gamma)}, \ p\to \infty.
\end{split}
\end{equation*}
In particular,
\begin{equation}
\label{AB-Sf}
S_f(p) =
\begin{cases}
\displaystyle O(\sqrt{p}\, \ln p),& {\rm \ if \ } \gamma=\dd,\\
\displaystyle
O(p^{\tau}),\ \tau=\frac{2r\mm + \dd}{4r\mm+2\gamma}>\frac{1}{2},
& {\rm \ if \ } \
\frac{\dd}{2}-r\mm<\gamma<\dd.
\end{cases}
\end{equation}

To establish \eqref{eq:same}, write
$$
S_1(p)=S_f(p)+pT\mathfrak{S}^{-r-(\gamma/\mm)}
S_{f,1}(p)+pTS_{f,2}(p),
$$
where
\begin{equation*}
\begin{split}
S_{f,1}(p)&=\sum_{k=1}^{\infty} f(k;p)-\int_1^{\infty} f(x;p)dx,\\
S_{f,2}(p)&=\sum_{k=1}^{\infty}
\left(\frac{1}{\lambda_k^{r+(\gamma/\mm)}
+\sqrt{\lambda_k^{2r+2(\gamma/\mm)}+2p\lambda_k^{\gamma/m}}}
 -\mathfrak{S}^{-r-(\gamma/\mm)}f(k;p)\right).
 \end{split}
 \end{equation*}
Then \eqref{eq:same} will follow from
\begin{align}
\label{same1}
pS_{f,1}(p)=o(S_f(p)), \ p\to \infty;\\
\label{same2}
pS_{f,2}(p)=o(S_f(p)), \ p\to \infty.
\end{align}

We have
$$
|S_{f,1}(p)|\leq 2\max_{x\geq 1} f(x;p),
$$
because, for fixed $p$, $0\leq f(x;p)\to 0$, $x\to +\infty$,
 and  the function $f$ has at most one critical point.
If $\gamma\geq 0$, then $\max_{x\geq 1} f(x;p)=f(1;p)=O(p^{-1/2})$,
and \eqref{same1} follows from \eqref{AB-Sf}.

If $\gamma<0$ [which is possible when $2r\mm>\dd$], then
$$
\arg\max_{x\geq 1} f(x;p)=O(p^{\dd/(4r\mm+2\gamma)}),\ p\to \infty,
$$
[by balancing $x^{4(r\mm+\gamma)/\dd}$
 and $px^{2\gamma/\dd}$],
so that, with  $\gamma <\dd$,
$$
p\,\max_{x\geq 1} f(x;p)= O(p^{2r\mm/(4r\mm+2\gamma)})=
o(S_f(p)),\ p\to \infty.
$$

To get a bound on $S_{2,f}$, note that
\begin{equation}
\label{eq:cf}
\left|\frac{\partial f(x;p)}{\partial x}\right| \leq C_f\frac{f(x;p)}{x},
\end{equation}
where  $C_f$ is a suitable constant independent of $p$. Together with
\eqref{S1-f}, inequality \eqref{eq:cf} implies
$$
|S_{f,2}(p)|\leq C_{f,2} \sum_{k=1}^{\infty} f(k;p)k^{-1/\dd},
$$
and the constant $C_{f,2}$ does not depend on $p$. By integral comparison,
$$
\sum_{k=1}^{\infty} f(k;p)k^{-1/\dd} \sim
\int_1^{\infty} f(x;p)x^{-1/\dd}\, dx,\ p\to \infty,
$$
and, similar to the derivation of \eqref{AB-Sf},
 $$
 p\int_1^{\infty} f(x;p)x^{-1/\dd}\, dx = o(S_f(p)),\ p\to \infty,
 $$
which implies \eqref{same2}.

The asymptotic \eqref{S1p-main} of $S_1(p)$ is now
proved; a more compact form of \eqref{S1p-main}
is
\begin{equation}
\label{pr-S1}
S_1(p)=
\begin{cases}
\displaystyle O(\sqrt{p}),& {\rm \ if \ } \gamma>\dd,\\
\displaystyle O(\sqrt{p}\, \ln p),& {\rm \ if \ } \gamma=\dd,\\
\displaystyle
O(p^{\tau}),\ \tau=\frac{2r\mm + \dd}{4r\mm+2\gamma}>\frac{1}{2},
& {\rm \ if \ } \
\frac{\dd}{2}-r\mm<\gamma<\dd.
\end{cases}
\end{equation}
It remains to establish \eqref{domin}. Recall that
$$
2S_2(p)=-\sum_{k=1}^{\infty} \ln \left(\frac{1+B_k}{2}\right);
$$
cf. \eqref{pr-AB}, \eqref{pr-1}, and \eqref{pr-S2-def}.
By definition, $0<B_k<1$, which means
$$
- \ln \left(\frac{1+B_k}{2}\right)=
\ln 2 -\ln (1+B_k) \leq 1-B_k,
$$
because the function $h(x)=1+\ln (1+x) - x $ is decreasing for $x\in [0,1]$
and $h(1)=\ln 2$. Next,
$$
B_k=\big(1+2p\lambda^{-(r+(\gamma/\mm))}\big)^{-1/2},
$$
and therefore
\begin{equation}
\label{Bk}
1-B_k\leq \frac{2p\lambda^{-(r+(\gamma/\mm))}}{1+2p\lambda^{-(r+(\gamma/\mm))}},
\end{equation}
 because
$$
1-(1+x)^{-1/2}\leq \frac{x}{1+x},\ x>0.
$$

Using \eqref{Weyl},
inequality \eqref{Bk} becomes
$
1-B_k\leq g_p(k),
$
where
$$
g_p(x)=\frac{C_g}{1+(x^{\mu}/p)},\ x\geq 0,
$$
$\displaystyle
\mu=\frac{4r\mm+2\gamma}{\dd},
$
and $C_g$ is a suitable constant independent of $p$. Then, by
integral comparison,
$$
2S_2(p)\leq g_p(1)+\int_0^{\infty} g_p(x)dx.
$$
Finally, by direct computation,
$$
g_p(1)+\int_0^{\infty} g_p(x)dx =
O(1)+C_gp^{1/\mu}\int_0^{\infty} \frac{dy}{1+y^{\mu}}=
O(p^{1/\mu})=O(p^{\dd/(4r\mm+2\gamma)}),\ p\to \infty;
$$
note that $\mu>1$ if $\gamma>\frac{\dd}{2}-r\mm$.
Comparing with \eqref{pr-S1}, we see that
$$
S_2(p)=o(S_1(p)), \ p\to \infty,
$$
 for all $\gamma>\frac{\dd}{2}-r\mm$.

To show that $S_3(p)=o(S_1(p)), \ p\to \infty$,
note that  \eqref{pr-S3-def} and
  inequality $\ln (1+x)\leq x$ imply
 $$
 2S_3(p)\leq \sum_{k=1}^{\infty} e^{-2A_k}
 \leq \sum_{k=1}^{\infty} e^{-2\lambda_k^r}
 =O(1), \ p\to \infty.
 $$

{\em This completes the proof of Theorem \ref{th:main}.}
\end{proof}

Comparing \eqref{SB-cylBM-1} and \eqref{SB-heat-1}, we see that,
for $\gamma>\dd$, the small ball behavior of both $W$ and $u$
is the same and, in a certain sense,  similar to the finite-dimensional case \eqref{SB-OU-fd}.
For $\frac{\dd}{2}-r\mm<\gamma<\dd$,
 the small ball rate for $W$ is bigger than the
small ball rate for $u$; in fact, in this range of $\gamma$, the
small ball rate is a decreasing function of $r\mm$.
 At $\gamma=\dd$, a kind of a
phase transition takes place.

For equation \eqref{example-PE}, we have
$$
\dd=r=\mm=\mathfrak{S}=1,\ \lambda_k=k^2,\
\|f\|_{\gamma}^2=\sum_{k=1}^{\infty} k^{2\gamma}f_k^2.
$$
Then
$$
\ln \mathbb{P}\left(\int_0^T  \|W(t)\|_{-\gamma}^2 dt \leq \varepsilon \right)
  \sim
  \begin{cases}
 \displaystyle
  - \frac{T^2}{8}
\left(\sum_{k=1}^{\infty} k^{-\gamma}\right)^2\
\varepsilon^{-1},& {\rm if} \ \gamma>1,\\
 & \\
\displaystyle
-\frac{T^2}{32} \, \varepsilon^{-1}\, |\ln \varepsilon|^2,& {\rm if} \ \gamma=1,\\
& \\
\displaystyle  - \mathfrak{C}_{\gamma}
\varepsilon^{-\varpi},& {\rm if} \ \frac{1}{2}<\gamma<1,
\end{cases}
$$
where
\begin{align*}
\varpi =\frac{1}{2\gamma-1},\
\mathfrak{C}_{\gamma}=(2\gamma-1)^{2\gamma}
\gamma^{-2\gamma\varpi}2^{(1-4\gamma)\varpi}
T^{2\varpi}C_{\gamma}^{2\gamma\varpi},\
{C}_{\gamma}=\int_0^{+\infty}
\frac{\ln \cosh(y)}{y^{1+(1/\gamma)}}\ dy.
\end{align*}
Similarly,
$$
 \ln \mathbb{P}\left(
 \int_0^T  \|u(t)\|_{-\gamma}^2 dt \leq \varepsilon \right)
  \sim
  \begin{cases}
  \displaystyle - \frac{T^2}{8}
\left(\sum_{k=1}^{\infty} k^{-\gamma}\right)^2\
\varepsilon^{-1},& {\rm \ if \ } \gamma>1,\\
& \\
\displaystyle  -\frac{T^2}{288}\,
 \varepsilon^{-1}\,
  |\ln \varepsilon|^2,& {\rm \ if \ } \gamma=1,\\
  & \\
 \displaystyle  -\mathfrak{C}_{\gamma}\, \varepsilon^{-\varpi},
 & {\rm \ if \ } -\frac{1}{2}< \gamma<1,
 \end{cases}
$$
where
\begin{align*}
\notag
&\mathfrak{C}_{\gamma}=\frac{\big((1-\tau)T
C_{\gamma}\big)^{1/(1-\tau)}}{2}\, (\varpi)^{\varpi},\
\notag
&  \tau=\frac{3}{4+2\gamma},\ \ \varpi=
   \frac{\tau}{1-\tau}=\frac{3}{2\gamma+1},\\
&C_{\gamma}= \int_0^{\infty}
\frac{dy}
{y^{2(1+\gamma)}+\sqrt{y^{4(1+\gamma)}
+y^{2\gamma}}}.
\end{align*}
In particular, taking  $\gamma=0$, we get a rather explicit logarithmic
asymptotic
\begin{equation}
\label{0pi-0}
\ln \mathbb{P}\left(
 \int_0^T \int_0^{\pi} u^2(t,x) dx dt \leq \varepsilon \right)
  \sim -\frac{81}{512} \,C_0^4 T^4\, \varepsilon^{-3},
\end{equation}
where
\begin{equation}
\label{C0}
C_0=\int_0^{\infty}\frac{dy}{y^2+\sqrt{y^4+1}} \approx 1.236.
\end{equation}

Figure 1 presents the small ball rates $\varpi$ when
$\gamma<1$ for the solution of \eqref{example-PE}
(bold curve) and for the underlying noise $W$.

\begin{figure}[ht]
  \label{Fig1}
\centering
\includegraphics[width=.95\linewidth]{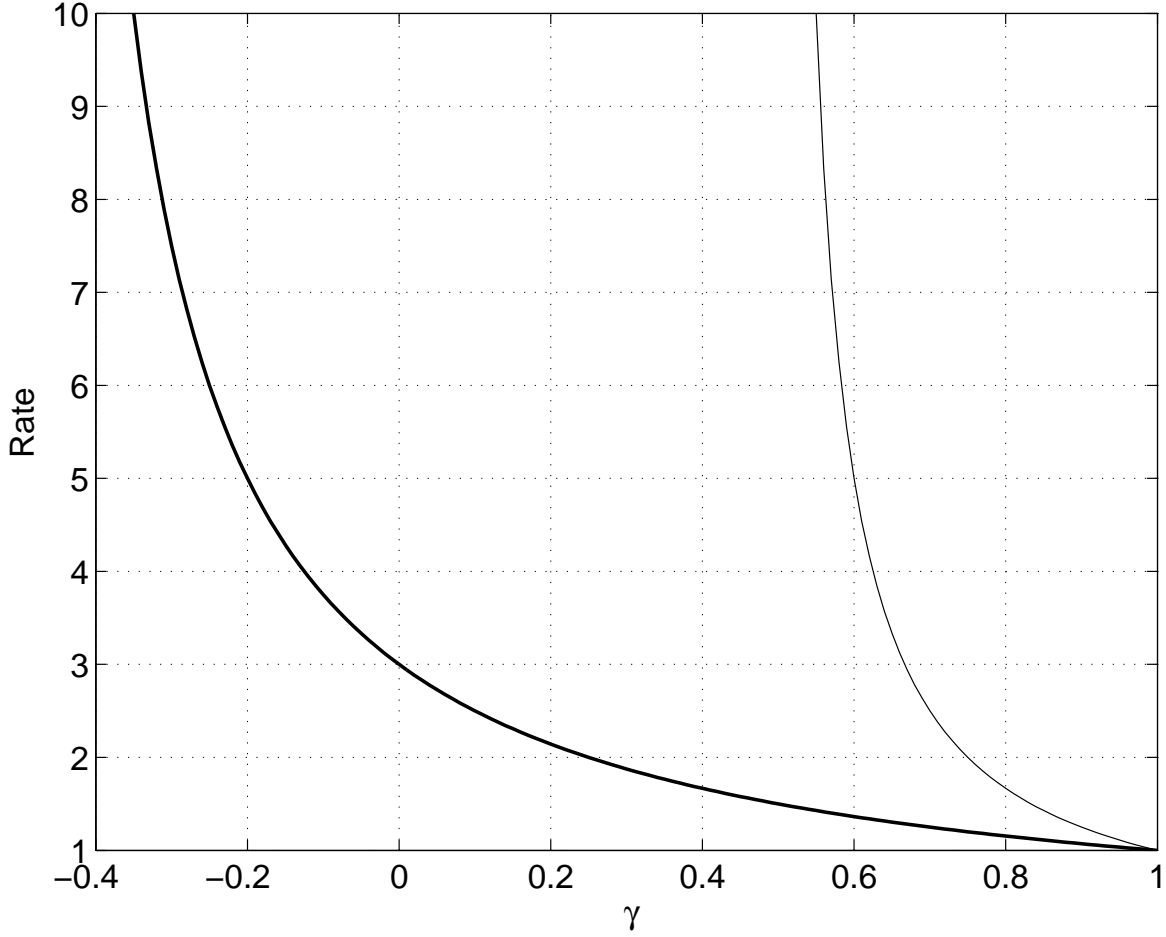}
\caption{Stochastic heat equation on $[0,\pi]$: Small ball
rate  in $H^{-\gamma}$, $\gamma<1$,
for the solution (bold curve) and the noise.}
\end{figure}

\section{Further Directions}

\subsection{Equivalent norms on $H^{\gamma}$}
If $\bld{a}=\{a_k,\ k\geq 1\}$, is a sequence of real numbers  such that
\begin{equation}
\label{eq-ak-basic}
c_1k\leq a_k\leq c_2k
\end{equation}
 for some $0<c_1\leq c_2$ and all $k$, then, by
 \eqref{Weyl},
$$
\|f\|_{\gamma; \bld{a}}^2=\sum_{k=1}^{\infty} a_k^{2\gamma/\dd} f_k^2
$$
defines an equivalent norm on $H^{\gamma}$. When $\gamma>\dd$,
we get immediate analogues of  Theorems \ref{prop:cBM-SB}
 and \ref{th:main}.

\begin{proposition}
Let $X$ denote either $u$ or $W$. For  $\gamma>\dd$,
\begin{equation}
\label{Sb-Na}
\ln \mathbb{P}\left(\int_0^T \|X(t)\|^2_{\gamma; \bld{a}}dt \leq \varepsilon\right)
\sim -\frac{T^2}{8} \left(\sum_{k=1}^{\infty} a_k^{-\gamma/\dd}\right)^2\,
\varepsilon^{-1},\ \varepsilon \to 0.
\end{equation}
\end{proposition}

\begin{proof}
Direct computations show that
$$
\ln \bE\exp\left(-p \int_0^T \|X(t)\|^2_{\gamma; \bld{a}}dt\right)
\sim -2^{-3/2} T\left(\sum_{k=1}^{\infty}
 a_k^{-\gamma/\dd}\right)\,\sqrt{p}, \ \  \ p\to \infty,
$$
either by \eqref{MainAsymptotic} [when $X=u$] or
by \eqref{cBM-p} [when $X=W$],
 and then \eqref{Sb-Na} follows from Theorem \ref{thm:tauberian}.
 \end{proof}

 When $\gamma\leq \dd$, analogues of Theorems \ref{prop:cBM-SB}
 and \ref{th:main} exist under an additional assumption about the
 numbers $a_k$.

\begin{proposition}
\label{prop:eqn}
Assume that
\begin{equation}
\label{ak-lim}
\lim_{k\to \infty} \frac{a_k}{k}=c_a>0.
\end{equation}
Then, as $\varepsilon \to 0$,
\begin{equation}
\label{cBM-SB-ak}
 \ln \mathbb{P}\left(\int_0^T  \|W(t)\|_{-\gamma;\bld{a}}^2 dt \leq \varepsilon \right)
  \sim
  \begin{cases}
\displaystyle
-\frac{T^2}{32c_a^2} \, \varepsilon^{-1}\, |\ln \varepsilon|^2,& {\rm if} \ \gamma=\dd,\\
& \\
\displaystyle  - \mathfrak{C}_{\gamma}
\varepsilon^{-\varpi},& {\rm if} \ \frac{\dd}{2}<\gamma<\dd,
\end{cases}
\end{equation}
where
\begin{align*}
\varpi &=\frac{\dd}{2\gamma-\dd},\
\mathfrak{C}_{\gamma}=(2q-1)^{2q}q^{-2q\varpi}2^{(1-4q)\varpi}
(Tc_a^{-\gamma/\dd})^{2\varpi}C_{\gamma}^{2q\varpi},\\
q&=\frac{\gamma}{\dd},\ {C}_{\gamma}=\int_0^{+\infty}
\frac{\ln \cosh(y)}{y^{1+(\dd/\gamma)}}\ dy.
\end{align*}

For the solution $u=u(t)$ of equation \eqref{PE},
\begin{equation}
\label{SB-heat-1-ak}
 \ln \mathbb{P}\left(
 \int_0^T  \|u(t)\|_{-\gamma;\bld{a}}^2 dt \leq \varepsilon \right)
  \sim
  \begin{cases}
\displaystyle  -\frac{T^2}{32c_a^2}\,
  \left(\frac{\dd}{\dd+2r\mm}\right)^2 \, \varepsilon^{-1}\,
  |\ln \varepsilon|^2,& {\rm \ if \ } \gamma=\dd,\\
  & \\
 \displaystyle  -\mathfrak{C}_{\gamma,\mm}\, \varepsilon^{-\varpi},
 & {\rm \ if \ } \frac{\dd}{2}-r\mm< \gamma<\dd,
 \end{cases}
\end{equation}
where
\begin{align*}
&\mathfrak{C}_{\gamma,\mm}=\frac{\big((1-\tau)\mathfrak{S}^{r(\gamma-\dd)/(\gamma+2r \mm)}
c_a^{-2\gamma\tau/\dd}T
C_{\gamma,\mm}\big)^{1/(1-\tau)}}{2}\, (\varpi)^{\varpi},\\
&  \tau=\frac{2r\mm+\dd}{4r\mm+2\gamma},\ \ \varpi=
   \frac{\tau}{1-\tau}=\frac{2r\mm +\dd}{2\gamma+2r\mm-\dd},\\
&C_{\gamma,\mm}= \int_0^{\infty}
\frac{dy}
{y^{2(r\mm+\gamma)/\dd}+\sqrt{y^{4(r\mm+\gamma)/\dd}
+y^{2\gamma/\dd}}}.
\end{align*}

\end{proposition}

\begin{proof}
By \eqref{Weyl} and \eqref{ak-lim},
$$
\lambda_k^{\gamma/\mm}\sim \mathfrak{S}^{\gamma/\mm}k^{2\gamma/\dd}
\sim c_a^{2\gamma/\dd}k^{2\gamma/\dd},\ k\to \infty.
$$
To derive \eqref{cBM-SB-ak}, it remains to replace $\mathfrak{S}$ in \eqref{SB-cylBM-1}
with $c_a^{2\mm/\dd}$. After a slightly more detailed analysis,
\eqref{SB-heat-1-ak} follows from \eqref{SB-heat-1} in a similar way.
Note that if $c_a=\mathfrak{S}^{\dd/(2\mm)}$, then
 \eqref{SB-heat-1-ak}  becomes \eqref{SB-heat-1}.
\end{proof}

In the special case $a_k=k$,
an alternative proof of Proposition \ref{prop:eqn} is possible using
the results from \cite[Example 2]{Li-JTP92}; for technical reasons,
such a proof is usually not possible under the general assumption
\eqref{ak-lim}.  Without \eqref{ak-lim} (that is, assuming only
\eqref{eq-ak-basic}), a precise logarithmic asymptotic of the small ball probabilities
 may not exist when $\gamma\leq \dd$, but the corresponding
 upper and lower bounds can still be derived.

\subsection{Noise with correlation in space}
A generalization of $\bld{W}$ is $\opr{Q}$-cylindrical
Brownian motion $\bld{W}^{\opr{Q}}$ defined by
$$
\bE \Big(\bld{W}^{\opr{Q}}(t,f)\bld{W}^{\opr{Q}}(s,g)\Big)
=\min(t,s) \int_G (\opr{Q}f)(x)g(x)dx,
$$
where $\opr{Q}$ is a non-negative symmetric operator on
$L_2(G)$. If
\begin{equation*}
\label{oprQ}
\opr{Q}\varphi_k=q_k^2\varphi_k,\ \ q_k>0,
\end{equation*}
then, similar to \eqref{cbm-sum}, we set
$$
W^{\opr{Q}}(t)=\sum_{k=1}^{\infty}
 \varphi_k \bld{W}^{\opr{Q}}(t,\varphi_k)
$$
and consider the equation
$$
\dot{u}(t)=\opr{A}^r u+\dot{W}^{\opr{Q}}(t),\ u(0)=0.
$$
If furthermore
\begin{equation*}
\label{qk}
q_k\sim c_qk^{\sigma},\ k\to \infty,
\end{equation*}
then the question about the asymptotic of
$\mathbb{P}\left(\int_0^T\|u(t)\|^2_{-\gamma}dt\leq \varepsilon\right), $
$\varepsilon\to 0$, is reduced to the corresponding question for the
solution of the equation
$$
\dot{v}(t)=\opr{A}^r v(t) +\dot{W}(t),
$$
where $v=\opr{Q}^{-1/2}u$. For example, if $q_k=\lambda_k^s$,
that is, $\opr{Q}=\opr{A}^{2s}$, $s\in \bR$,  then
$$
\mathbb{P}\left(\int_0^T\|u(t)\|^2_{-\gamma}dt\leq \varepsilon\right)
=
\mathbb{P}\left(\int_0^T\|v(t)\|^2_{-\gamma+2\mm s}dt\leq \varepsilon\right).
$$

\subsection{Non-zero initial condition}
Assume that the initial condition
$u(0)$ in \eqref{PE} is independent of
$W$ and has the form
$$
u(0)=\sum_{k=1}^{\infty} u_k(0)\varphi_k,
$$
where $u_k(0)$ are independent Gaussian random variables
with mean $\mu_k$ and variance $\sigma_k^2$. To ensure
\eqref{PE-reg1}, condition
\begin{equation}
\label{IC-reg}
\sum_{k=1}^{\infty}k^{-2\gamma/\dd}\big( \mu_k^2+\sigma_k^2\big) <\infty
\end{equation}
must hold for all $\gamma >\dd/2$.

In the finite-dimensional case \eqref{OU-fd}, it is known
\cite[Theorem 4.5]{SVL-MM} that the initial condition
may affect the small ball constant but not the small ball rate:
if $\bld{x}(0)$ is a Gaussian
random vector independent of $\bld{w}$, then
$$
\ln \mathbb{P}\left(\int_0^T |\bld{x}(t)|^2dt\leq \varepsilon\right)
\sim - \mathfrak{C}\, \varepsilon^{-1}, \ \varepsilon \to 0,
$$
where $\mathfrak{C}$ may depend on the mean and covariance of
$\bld{x}(0)$. In particular, if the covariance matrix of $\bld{x}(0)$ is non-singular,
then $\mathfrak{C}={T^2}/{8}$, that is, the initial condition does not
change the small ball asymptotic at the logarithmic level. The corresponding
results in the infinite-dimensional case are as follows.

 \begin{proposition}
 \label{prop:inf-dim-IC}
 Assume that  \eqref{IC-reg} holds.
\begin{enumerate}
\item If
 \begin{equation}
 \label{IC-ndgv}
 \inf_{k}\sigma_k=\sigma_0>0,
 \end{equation}
  then \eqref{SB-heat-1} holds.
  \item If  $\gamma>\dd$ and $\sigma_k=0$ for all $k\geq 1$,
 then
  \begin{equation}
  \label{IC-zeroSG}
\ln  \mathbb{P}\left(\int_0^T\|u(t)\|^2_{-\gamma}dt\leq \varepsilon\right)
 \sim -
  \frac{1}{8}
\left(\sum_{k=1}^{\infty}
\big(T+ \mu_k^2)\lambda_k^{-\gamma/(2\mm)}\right)^2
\, \varepsilon^{-1},\ \varepsilon\to 0.
  \end{equation}
  \end{enumerate}
 \end{proposition}

 \begin{proof}  The idea is to trace the contributions of the
 initial condition throughout the proof of
 Theorem \ref{th:main}. In particular, our objective is the asymptotic of
  $\ln \bE \exp\left(-p\int_0^T\|u(t)\|_{-\gamma}^2dt\right)$
  as $p\to \infty$.

 By \eqref{OU-Psi-IC},  the initial condition contributes
an extra multiplicative term  
$$
\exp\left(-S_{0,1}(p)-S_{0,2}(p)\right),
$$
 where
\begin{align*}
 S_{0,1}(p)&=\sum_{k=1}^{\infty}
 \frac{\mu_k^2\psi_{0,k}(p)}{1+2\sigma_k^2\psi_{0,k}(p)},\ \
 2S_{0,2}(p)=\sum_{k=1}^{\infty}
  \ln \big(1+2\sigma_k^2\psi_{0,k}(p)\big),\\
 \psi_{0,k}(p)&
 =\frac{p\lambda_k^{-\gamma/\mm}}
 {\lambda_k^r+\sqrt{\lambda_k^{2r}+2p\lambda_k^{-\gamma/\mm}}}.
 \end{align*}
 Also recall that, with zero initial condition, the dominant term is
 $$
 S_1(p)=pT\sum_{k=1}^{\infty}
  \frac{\lambda_k^{-\gamma/\mm}}{\lambda_k^r+
  \sqrt{\lambda_k^{2r}+2p\lambda_k^{-\gamma/\mm}}}.
  $$

 Under \eqref{IC-ndgv},
 $$
 S_{0,1}(p)\leq \sum_{k=1}^{\infty}
 \frac{p\mu_k^2\lambda_k^{-\gamma/\mm}}
 {\lambda_k^r+p\sigma_0^2\lambda_k^{-\gamma/\mm}}=o(S_1(p)), \
 p\to \infty.
 $$
 Similarly, $S_{0,2}(p)
 =o(S_1(p)), \  p\to \infty. $ In other words, if the variance of the initial
 condition is strictly non-degenerate, then the initial condition does not
 affect the asymptotic of
 $\ln \bE \exp\left(-p\int_0^T\|u(t)\|_{-\gamma}^2dt\right)$
 as $p\to \infty$.

 If $\gamma>\dd$ and $\sigma_k=0$, then
\eqref{IC-reg} implies
$$
\sum_{k=1}^{\infty} \lambda_k^{-\gamma/(2\mm)}\mu_k^2<\infty,
$$
and, similar to the proof of \eqref{MainAsymptotic},
$$
\ln \bE \exp\left(-p\int_0^T\|u(t)\|_{-\gamma}^2dt\right)
\sim -\frac{1}{\sqrt{2}}
\sum_{k=1}^{\infty}
\big(T+ \mu_k^2)\lambda_k^{-\gamma/(2\mm)}.
$$
After that, Theorem \ref{thm:tauberian} implies
\eqref{IC-zeroSG}.
 \end{proof}

 The stationary case requires special consideration; cf. \cite{fatalov08} for one-dimensional OU process.
 \begin{proposition}
 Assume that $\mu_k=0$ and
 $\sigma_k^2=(2\lambda_k^r)^{-1}$. Then
\eqref{SB-heat-1} holds.
\end{proposition}

\begin{proof}
Even though direct application of Proposition \ref{prop:inf-dim-IC}(1)
is not possible because now  $\inf_{k}\sigma_k=0$
and \eqref{IC-ndgv} fails, very little changes in the actual proof:
 with only the $S_{0,2}(p)$ term present, we see that
$S_{0,2}(p)=o(S_1(p)), \  p\to \infty, $ still holds, that is, the initial condition
does not
affect the asymptotic of
$\ln \bE \exp\left(-p\int_0^T\|u(t)\|_{-\gamma}^2dt\right)$.
\end{proof}

If $\gamma\leq \dd$, then
   initial condition can affect the small ball rate.
For example,  assume that the initial condition in \eqref{example-PE}
is non-random [$\sigma_k=0$] and $\mu_k=\sqrt{\ln k}$, so that
\eqref{IC-reg} holds.
Using  \eqref{OU-Psi-IC}, \eqref{MainAsymptotic}, and
\eqref{ld-A},
$$
\ln \bE \exp\left(-p \int_0^T\int_0^{\pi} u^2(t,x)dxdt\right)\sim
-p\sum_{k=1}^{\infty} \frac{T+\ln k}{k^2+\sqrt{k^4+2p}},\ p\to \infty.
$$
Similar to derivation of \eqref{S1p-main},
$$
\sum_{k=1}^{\infty} \frac{T+\ln k }{k^2+\sqrt{k^4+2p}}
\sim
\frac{\ln p}{4(2p)^{1/4}} \int_0^{\infty} \frac{dy}{y^2+\sqrt{y^4+1}},
$$
that is,
$$
\ln \bE \exp\left(-p \int_0^T\int_0^{\pi} u^2(t,x)dxdt\right)\sim
-2^{-9/4}C_0\, p^{3/4}\ln p,
$$
with $C_0$ from \eqref{C0}. By Theorem \ref{thm:tauberian1}
with $\alpha=2^{-9/4}C_0$, $\beta=1$, $\tau=3/4$,
$$
\ln \mathbb{P}\left(
 \int_0^T \int_0^{\pi} u^2(t,x) dx dt \leq \varepsilon \right)
  \sim -\frac{27C_0^4}{2^{17}}  \,\varepsilon^{-3}|\ln \varepsilon |^{4},
$$
which is very different from \eqref{0pi-0}: the rate has an additional logarithmic term and the constant does not depend on $T$.

\subsection{Other types of parabolic equations}
Consider a linear operator $\opr{A}$ on
a separable Hilbert space $H$. If $\opr{A}$ is symmetric and
 has a pure point spectrum  $0<\lambda_1\leq \lambda_2\leq
\ldots,$ and the corresponding eigenfunctions
$\varphi_k$ form an orthonormal basis in $H$, then
all the constructions from Section \ref{sec:DSPE} can be
repeated, and an analog of Theorem \ref{th:main} can be
stated and proved
for the  evolution equation
\begin{equation}
\label{eq:space}
u_t(t)+\opr{A}u(t)=\dot{W}(t),\ u(0)=0,
\end{equation}
where $\dot{W}$ is a cylindrical Brownian motion on $H$.

The details depend on the asymptotic behavior of $\lambda_k$ as
$k\to \infty$.
For example, consider the equation
$$
u_t(t,x)=\frac{1}{2}\big(u_{xx}(t,x)-x^2u(t,x)-u(t,x)\big)
+\dot{W}(t,x),\ t>0,\ x\in \bR,\ u(0,x)=0.
$$
Then
$$
H=L_2(\bR), \quad \opr{A}=\frac{1}{2}\left(
-\frac{\partial^2}{\partial x^2}+x^2+1
\right),
$$
and
$$
\lambda_k=k,\ \varphi_k(x)=(-1)^k\frac{1}{\sqrt{2^kk!}}\pi^{-1/4}
e^{x^2/2}\frac{d^k}{dx^k}e^{-x^2};
$$
cf. \cite[Section 1.4]{Obata}.
Since operator $\opr{A}$ has order 2, we
define
$$
H^{\gamma}=\opr{A}^{\gamma/2}L_2(\bR),\
 \|f\|_{\gamma}^2=\sum_{k=1}^{\infty} k^{\gamma}f_k^2.
$$
Recall that the norm in the traditional
Sobolev space on $\bR$ is
$$
|\![f]\!|_{\gamma}^2=\int_{-\infty}^{+\infty}
|\hat{f}(y)|^2(1+y^2)^{\gamma}dy;
$$
 $\hat{f}$ is the Fourier transform of $f$.
 In particular, it follows that
 $$
 \bE|\![W(t)]\!|_{\gamma}^2=\infty
 $$
 for every $\gamma\in \bR$ and every $t>0$
 [roughly speaking, because $\sum_{k} \varphi_k^2(x)=\delta(x)$ and
 each $\varphi_k$ is an eigenfunction of the Fourier transform], and
 consequently the solution of
$v_{t}=v_{xx}+\dot{W}(t,x)$, $x\in \bR$, does not
belong to any traditional Sobolev space on $\bR$. On the other hand, similar to
Propositions \ref{prop:WN} and \ref{prop:HE},
 $$
W\in L_2\Big(\Omega\times (0,T);H^{-\gamma}\Big),\
u\in L_2\Big(\Omega\times (0,T);H^{-\gamma+1}\Big),\ \gamma>1;
$$
$u$ is the solution of \eqref{eq:space}.
The corresponding small ball asymptotics can also be derived.

\begin{proposition}
The following relations hold as $\varepsilon\to 0$:
$$
\ln \mathbb{P}\left(\int_0^T  \|W(t)\|_{-\gamma}^2 dt \leq \varepsilon \right)
  \sim
  \begin{cases}
 \displaystyle
  - \frac{T^2}{8}
\left(\sum_{k=1}^{\infty} k^{-\gamma/2}\right)^2\
\varepsilon^{-1},& {\rm if} \ \gamma>2,\\
 & \\
\displaystyle
-\frac{T^2}{32} \, \varepsilon^{-1}\, |\ln \varepsilon|^2,& {\rm if} \ \gamma=2,\\
& \\
\displaystyle  - \mathfrak{C}_{\gamma}
\varepsilon^{-1/(\gamma-1)},& {\rm if} \ 1<\gamma<2,
\end{cases}
$$
where
\begin{align*}
\mathfrak{C}_{\gamma}=(\gamma-1)^{\gamma}
\gamma^{-2\gamma\varpi}2^{(1-2\gamma)\varpi}
T^{2\varpi}C_{\gamma}^{\gamma\varpi},\
\varpi =\frac{1}{\gamma-1},\
{C}_{\gamma}=\int_0^{+\infty}
\frac{\ln \cosh(y)}{y^{1+(2/\gamma)}}\ dy,
\end{align*}
and
\begin{equation*}
\label{SB-space}
 \ln \mathbb{P}\left(
 \int_0^T  \|u(t)\|_{-\gamma}^2 dt \leq \varepsilon \right)
  \sim
  \begin{cases}
  \displaystyle - \frac{T^2}{8}
\left(\sum_{k=1}^{\infty} k^{-\gamma}\right)^2\
\varepsilon^{-1},& {\rm \ if \ } \gamma>2,\\
& \\
\displaystyle  -\frac{T^2}{128}\,
 \varepsilon^{-1}\,
  |\ln \varepsilon|^2,& {\rm \ if \ } \gamma=2,\\
  & \\
 \displaystyle  -\mathfrak{C}_{\gamma}\, \varepsilon^{-2/\gamma},
 & {\rm \ if \ } 0< \gamma<2,
 \end{cases}
\end{equation*}
where
\begin{align*}
\notag
&\mathfrak{C}_{\gamma}=\frac{\big((1-\tau)T
C_{\gamma}\big)^{1/(1-\tau)}}{2}\, (\varpi)^{\varpi},\ \ \
  \tau=\frac{2}{2+\gamma},\ \ \varpi=
   \frac{\tau}{1-\tau}=\frac{2}{\gamma},\\
&C_{\gamma}= \int_0^{\infty}
\frac{dy}
{y^{1+\gamma}+\sqrt{y^{2+2\gamma}
+y^{\gamma}}}.
\end{align*}
\end{proposition}

\begin{proof}
The case of $W$ follows from the asymptotic of
$$
 \sum_{k=1}^{\infty}
\ln \cosh\left(T\sqrt{2pk^{-\gamma}}\;\right),\ p\to \infty,
$$
similar to the proof of Theorem  \ref{prop:cBM-SB}.

The case of $u$  follows the same steps as the
 proof of  Theorem \ref{th:main}.
\end{proof}

\section{Summary}
Consider the equation
$$
\dot{u}(t)+\opr{A}^r u=\dot{W}^{\opr{Q}}(t),\ u(0)=0,
$$
with $r>0$,
and assume that the positive-definite operators
 $\opr{A}$ and $\opr{Q}$ commute,  have
 purely point  spectrum
and act in a scale $H^{\gamma}$ of Hilbert spaces,
and
$$
\gamma_0=\inf\{s>0\ |\  \mathfrak{i}\circ\opr{Q}:H^0\to H^{-s}\ \
{\rm \ is \ trace\ class}\} <\infty,
$$
where $\mathfrak{i}$ is the embedding operator.
 If $\gamma>\gamma_0$, then the logarithmic asymptotic of the
 small ball probabilities is similar to finite-dimensional case
 \eqref{SB-OU-fd}:
 $$
 \ln \mathbb{P}\left(\int_0^T \|X(t)\|_{-\gamma}^2dt\leq \varepsilon
 \right)\sim -\frac{\big(T\,\mathrm{trace}(\mathfrak{i}\circ\opr{Q})\big)^2}{8}\,
 \varepsilon^{-1},\
 \varepsilon\to 0,
 $$
for both $X=u$ and $X=W^{\opr{Q}}$.
 Infinite-dimensional effects  appear when
 $\gamma\leq \gamma_0$: the small ball rate now depends on
 $\gamma$ and can be arbitrarily large, whereas the small ball constant
  depends on the operator $\opr{A}$.


\begin{thebibliography}{10}

\bibitem{RegVar}
N.~H. Bingham, C.~M. Goldie, and J.~L. Teugels.
\newblock {\em Regular variation}, volume~27 of {\em Encyclopedia of
  Mathematics and its Applications}.
\newblock Cambridge University Press, Cambridge, 1987.

\bibitem{DaPr1-2}
G.~Da~Prato and J.~Zabczyk.
\newblock {\em Stochastic equations in infinite dimensions}, volume 152 of {\em
  Encyclopedia of Mathematics and its Applications}.
\newblock Cambridge University Press, Cambridge, second edition, 2014.

\bibitem{fatalov08}
V.~R. Fatalov.
\newblock Exact asymptotics of small deviations for the stationary
  {O}rnstein-{U}hlenbeck process and some {G}aussian diffusions in the
  {$L^p$}-norm, {$2\leq p\leq\infty$}.
\newblock {\em Problemy Peredachi Informatsii}, 44(2):75--95, 2008.

\bibitem{Khosh}
D.~Khoshnevisan.
\newblock {\em Multiparameter processes: {A}n introduction to random fields}.
\newblock Springer, New York, 2002.

\bibitem{Khosh-SPDE}
D.~Khoshnevisan.
\newblock {\em Analysis of stochastic partial differential equations}, volume
  119 of {\em CBMS Regional Conference Series in Mathematics}.
\newblock AMS, Providence, RI, 2014.

\bibitem{Koncz}
K.~Koncz.
\newblock On the parameter estimation of diffusional type processes with
  constant coefficients (elementary {G}aussian processes).
\newblock {\em Anal. Math.}, 13(1):75--91, 1987.

\bibitem{Li-JTP92}
W.~V. Li.
\newblock Comparison results for the lower tail of {G}aussian seminorms.
\newblock {\em J. Theoret. Probab.}, 5(1):1--31, 1992.

\bibitem{lishao01}
W.~V. Li and Q.-M. Shao.
\newblock Gaussian processes: inequalities, small ball probabilities and
  applications.
\newblock In D.~N. Shanbhag and C.~R. Rao, editors, {\em Stochastic processes:
  theory and methods}, volume~19 of {\em Handbook of Statist.}, pages 533--597.
  North-Holland, Amsterdam, 2001.

\bibitem{LSh-II}
R.~S. Liptser and A.~N. Shiryaev.
\newblock {\em Statistics of random processes, {II}: {A}pplications}, volume~6
  of {\em Applications of Mathematics}.
\newblock Springer-Verlag, Berlin, 2001.

\bibitem{SVL-MM}
S.~V. Lototsky and M.~Moers.
\newblock Large-time and small-ball asymptotics for quadratic functionals of
  {G}aussian diffusions.
\newblock {\em Asymptot. Anal.}, 95(3--4):345--374, 2015.

\bibitem{SmallBall-StochWave}
A.~Martin.
\newblock Small ball asymptotics for the stochastic wave equation.
\newblock {\em J. Theoret. Probab.}, 17(3):693--703, 2004.

\bibitem{Obata}
N.~Obata.
\newblock {\em White noise calculus and {F}ock space}, volume 1577 of {\em
  Lecture Notes in Mathematics}.
\newblock Springer, Berlin, 1994.

\bibitem{Roz}
B.~L. Rozovski{\u\i}.
\newblock {\em Stochastic evolution systems}, volume~35 of {\em Mathematics and
  its Applications (Soviet Series)}.
\newblock Kluwer, Dordrecht, 1990.

\bibitem{SafVas}
Yu. Safarov and D.~Vassiliev.
\newblock {\em The asymptotic distribution of eigenvalues of partial
  differential operators}, volume 155 of {\em Translations of Mathematical
  Monographs}.
\newblock American Mathematical Society, Providence, RI, 1997.

\bibitem{Syt1}
G.~N. Sytaya.
\newblock On some asymptotic representation of the {G}aussian measure in a
  {H}ilbert space.
\newblock {\em Theory of Stochastic Processes}, 2:93--104, 1974.

\bibitem{Walsh}
J.~B. Walsh.
\newblock An introduction to stochastic partial differential equations.
\newblock In P.~L. Hennequin, editor, {\em Ecole d'\'et\'e de Probabilit\'es de
  Saint-Flour, XIV, Lecture Notes in Mathematics}, volume 1180, pages 265--439,
  Springer, 1984.

\end{thebibliography}

\end{document}